\documentclass[12pt,a4]{article}
\usepackage{amsmath, amsthm, amssymb}
\usepackage[all]{xy}
\usepackage{amscd}
\newtheorem{theorem}{Theorem}[section]

\newtheorem{corollary}[theorem]{Corollary}

\everymath{\displaystyle}

\begin{document}
\bibliographystyle{plain}
\pagestyle{plain} \pagenumbering{arabic}
\date{}
\newpage
\title{\textbf{\small{On The Cohomological Dimension of Local Cohomology Modules  }}}
\author{\small{Vahap Erdo\u{g}du and Tu\u{g}ba Y\i{ı}ld\i{ı}r\i{ı}m\footnote{ Corresponding author}}}
\maketitle
\begin{center}
\small{\textit{Dedicated to the memory of Alexander Grothendieck}}
\end{center}
\begin{center}
\small{{Department of Mathematics} \\
{Istanbul Technical University}\\
{Maslak, 34469, Istanbul, Turkey}}
\end{center}
\maketitle

\begin{abstract}
Let $R$ be a Noetherian ring, $I$ an ideal of $R$ and $M$ an $R$-module with $\operatorname{cd}(I,M)=c$. In this article,  we first show that  there exists a descending chain of ideals   $I=I_c\supsetneq I_{c-1}\supsetneq  \cdots \supsetneq I_0$ of $R$ such that for each $0\leq i\leq c-1$, $\operatorname{cd}(I_i,M)=i$ and that the top local cohomology module $\operatorname{H}^i_{I_i}(M)$ is not Artinian. We then give sufficient conditions for a non-negative integer $t$ to be a lower bound for $\operatorname{cd}(I,M)$ and use this to conclude that  in non-catenary Noetherian local integral domains, there exist prime ideals that are not set theoretic complete intersection. Finally, we set conditions which determine whether or not a top local cohomology module is Artinian.
\end{abstract}
 {\bf \emph{Keywords}:}\emph{ Top local cohomology modules, Cohomological dimensions, Set theoretic complete intersections. 
}
\section{Introduction}\label{intro}
Throughout, $R$ denote a commutative Noetherian ring with unity, $I$ an ideal of $R$. For an $R$-module $M$, the $i$-th local cohomology module of $M$ with support in $I$ is defined as
\begin{equation*} 
\operatorname{H}^{i}_{I}(M)=\lim\limits_{\rightarrow} \operatorname{Ext}_{R}^{i}(R/I^{n},M).
\end{equation*}
For details about the local cohomology modules, we refer the reader to \cite{BS} and \cite{R.Hartshorne}.\\
One of the important invariant related to local cohomology modules is \textit{the cohomological dimension of M with respect to I}, denoted by $\operatorname{cd}(I,M)$, and  defined as:
 \begin{equation*} \operatorname{cd}(I,M) = \operatorname{sup}\  \{ i \in\mathbb{N}\ |\ \operatorname{H}^i_I(M)\neq 0 \}.
\end{equation*}
If $M=R$, we write $cd(I)$ instead of $cd(I,R)$.\\
There are two interesting questions related to local cohomology modules, the first one is to  determine the lower and upper bounds for $\operatorname{cd}(I,M)$ and the second one is to determine whether or not  $\operatorname{H}^i_I(M)$ is Artinian(see e.g. \cite{Aghapournahr}, \cite{Dibaei}, \cite{Conn}, \cite{Hartshorne}, \cite{Huneke} and \cite{Huneke-Lyubeznik}).\\
Our results in this regard are as follows:\\
In section 2, we show that for an $R$-module $M$ with $\operatorname{cd}(I,M)=c$, there is a descending chain of ideals 
\begin{equation*} I=I_c\supsetneq I_{c-1}\supsetneq  \cdots \supsetneq I_0
\end{equation*}
of $R$ such that for each $0\leq i\leq c$, $\operatorname{cd}(I_i,M)=i$, and that the top local cohomology module $\operatorname{H}^i_{I_i}(M)$ is not Artinian. \\
In section 3, we prove a result that gives a sufficient condition for an integer to be a lower bound for the cohomological dimension, $\operatorname{cd}(I,M)$, of $M$ at $I$. One of the important conclusion of this result is that over a Noetherian local ring $(R,\mathfrak m) $, for a finitely generated $R$-module $M$ of dimension $n$ and an ideal $I$ of $R$ with $\dim(M/IM)=d\geq 1$, $n-d$ is a lower bound for $\operatorname{cd}(I,M)$ and if moreover $\operatorname{cd}(I,M)=n-d$, then $
\operatorname{H}^d_{\mathfrak m} (\operatorname{H}^{n-d}_I(M)) \cong \operatorname{H}^{n}_{\mathfrak m}(M).$  As an application of this result, we show that in non-catenary Noetherian local integral domains, there exist prime ideals that are not set theoretic complete intersection.\\
In section 4, we examine the Artinianness and non-Artinianness of top local cohomology modules.
\section{Descending Chains With Successive Cohomological Dimensions}\label{sec:1}
In this section, we prove the existence of descending chains of ideals and locally closed sets with successive cohomological dimensions. The main result of this section is the following:
\begin{theorem} \label{theorem2.1}Let $R$ be a Noetherian ring, $I$ an ideal of $R$ and $M$ an $R$-module with  $\operatorname{cd}(I,M)=c>0.$ Then there is a descending chain of ideals 
\begin{equation*} I=I_c\supsetneq I_{c-1}\supsetneq  \cdots \supsetneq I_0
\end{equation*}
such that $\operatorname{cd}(I_i,M)=i$ for all $i$, $0\leq i\leq c$. Moreover $\operatorname{H}^i_{I_i}(M)$ is not Artinian for all $i$, $0\leq i\leq c-1$.
\end{theorem}
\begin{proof} Consider the set 
\begin{equation*} \mathbb{S}=\ \{ \ J\subsetneq I\ |\ \operatorname{cd}(J,M)< c\ \}.
\end{equation*}
Clearly, the zero ideal belongs to $\mathbb S$ and so $\mathbb S$ is a non-empty subset of ideals of $R$. Since $R$ is Noetherian, $\mathbb S$ has a maximal element, say $I_{c-1}$. We claim that $\operatorname{cd}(I_{c-1},M) =c-1$. To prove this, let $x\in I\setminus I_{c-1}$ and so $ I_{c-1}+Rx\subseteq I$. But then it follows from the maximality of $ I_{c-1}$ in $\mathbb{S}$ and Remark 8.1.3 of  \cite{BS} that 
\begin{equation*} c\leq \operatorname{cd}(I_{c-1}+Rx,M)\leq \operatorname{cd}(I_{c-1},M)+1<c+1.
\end{equation*}
Hence $\operatorname{cd}(I_{c-1}+Rx,M)=c$. Now consider the exact sequence
\begin{equation*}\xymatrix{\cdots \ar[r]& {(\operatorname{H}^{c-1}_{I_{c-1}} (M))}_x\ar[r]& \operatorname{H}^c_{I_{c-1}+Rx}(M)\ar[r]&{\operatorname{H}^{c}_{I_{c-1}} (M)}=0}.
\end{equation*}
Since  $\operatorname{H}^c_{I_{c-1}+Rx}(M)$ is nonzero, it follows that ${(\operatorname{H}^{c-1}_{I_{c-1}} (M))}_x$ is nonzero, then so is $\operatorname{H}^{c-1}_{I_{c-1}} (M)$. Therefore the claim follows.\\
Iterating this argument, one can obtain a descending chain of ideals, as desired.\\
For the second part, let $x\in I_{i+1}\setminus I_i$ and consider the ideal $I_{i}+Rx$. Then it follows from the construction of the ideal $I_i$ that $\operatorname{cd}({{I_i}+Rx},M)=i+1$. Now $\operatorname{H}^i_{I_i}(M)$ is non-Artinian follows from Corollary 4.1 of \cite{Dibaei}.
\end{proof}
Recall that a subspace $Z $ of a topological space $X$ is said to be \textit{ locally closed}, if it is the intersection of an open and a closed set. 
Let $X$ be a topological space, $Z\subseteq X$ be a locally closed subset of $X$ and let $F$ be an abelian sheaf on $X$. Then the $i^{th }$ local cohomology group of $F$  with support in $Z$ is denoted by $H^{i}_Z(X,F)$. We refer the reader to \cite{R.Hartshorne} and \cite{Ly} for its definition and details.\\
If, in particular, $X=Spec(R)$ is an affine scheme, where $R$ is a commutative Noetherian ring, and $F=M^{\sim}$ is the quasi coherent sheaf on $X$ associated to an $R$-module $M$, we write $H^{i}_Z(M)$ instead of $H^{i}_Z(X,M^{\sim})$.\\
The following corollary may be considered as an easy application of our result above.
\begin{corollary} Let $R$ be a Noetherian ring, $M$ an $R$-module and $I$ an ideal of $R$ such that $cd(I,M)=c>1.$ Then there is a descending chain of locally closed sets 
\begin{equation*} T_{c-1}\supsetneq T_{c-2}\supsetneq  \cdots \supsetneq T_1
\end{equation*}
in Spec(R) such that $cd(T_i,M)=i$ for all $1\leq i\leq c-1$.
\end{corollary}
\begin{proof}  Let $I$ be an ideal of $R$ with $cd(I,M)=c>1.$ Then it follows from Theorem \ref{theorem2.1} that there is a descending chain of ideals 
\begin{equation*} I=I_c\supsetneq I_{c-1}\supsetneq  \cdots \supsetneq I_1 \supsetneq I_0
\end{equation*}
such that $cd(I_i,M)=i$ for all $0\leq i\leq c$. Let now $U_i=V(I_i)$ and define the locally closed sets $T_i:=U_1\setminus U_{i+1}$. Then it is easy to see that
\begin{equation*} T_{c-1}\supsetneq T_{c-2}\supsetneq  \cdots \supsetneq T_1.
\end{equation*}
On the other hand, it follows from Proposition 1.2 of \cite{Ly} that there is a long exact sequence, 
\begin{equation*} \xymatrix{\cdots \ar[r]& H^{j}_{U_1}(M)\ar[r]&{H^{j}_{T_i}(M)}\ar[r]& H^{j+1}_{U_{i+1}} (M)\ar[r]& H^{j+1}_{U_1}(M)\ar[r]& \cdots}
\end{equation*}
As $H^{j}_{U_i}(M)\cong H^{j}_{I_i}(M)$ for all $1\leq i\leq c-1$ and for all $j\geq 0$, it follows from the above long exact sequence that $cd(T_i,M)=i.$
\end{proof}


\section{Lower Bound For Cohomological Dimension}\label{sec:3}
The main purpose of this section is to establish a lower bound for cohomological dimension and, in this regard, we prove the following theorem which gives a sufficient condition for an integer $t$ to be a lower bound for $\operatorname{cd}(I,M)$.  
\begin{theorem}
\label{theorem3.1}Let $R$ be a Noetherian ring, $M$ an $R$-module (not necessarily finitely generated) and $I$ an ideal of $R$ with $\dim(R/{I+\operatorname{Ann}M})=d$. Let $t\geq 0$ be an integer. If there exists an ideal $J$ of $R$ such that $\operatorname{H}^{d+t}_{I+J}(M)\neq 0$, then $t$ is a lower bound for $\operatorname{cd}(I,M)$. Moreover, if $\operatorname{cd}(I,M)=t$, then
 \begin{equation*}
\operatorname{H}^d_J (\operatorname{H}^{t}_I(M)) \cong \operatorname{H}^{d+t}_{I+J}(M)
\end{equation*}
and $\dim\operatorname{Supp}(\operatorname{H}^{t}_I(M))= d$.
\end{theorem}
\begin{proof} Consider the Grothendieck's spectral sequence
\begin{equation*} E^{p,q}_2=\operatorname{H}^p_{J}(\operatorname{H}^q_I(M))  \Longrightarrow \operatorname{H}^{p+q}_{I+J}(M)
\end{equation*}
and look at the stage $p+q=n$. Since $\operatorname{Supp}(\operatorname{H}^q_I(M))\subseteq V(I)\cap \operatorname{Supp}(M)\subseteq V(I+\operatorname{Ann}M)$, $\dim\operatorname{Supp}(\operatorname{H}^q_I(M))\leq d$ for all $q$. Therefore it follows from Grothendieck's vanishing theorem that  for all $p>d,$ $E^{p,d+t-p}_2=0$. But then since from the hypothesis $\operatorname{H}^{d+t}_{I+J}(M)$ does not vanish, there is at least one $p\leq d$ such that
\begin{equation*} E^{p,d+t-p}_2=\operatorname{H}^p_{J}(\operatorname{H}^{d+t-p}_I(M)) \neq 0.
\end{equation*}
Hence $\operatorname{H}^{d+t-p}_I(M)\neq 0$ and so $\operatorname{cd}(I,M)\geq d+t-p\geq t.$\\
If, in particular, $\operatorname{cd}(I,M)=t$, then $E^{p,q}_2=0$ for all $q>t$. Now from the subsequent stages of the spectral sequence
\begin{equation*}
\xymatrix{E^{d-k,t+k-1}_k\ar[r]&E^{d,t}_k\ar[r]&E^{d+k,t-k+1}_k} 
\end{equation*}
and the fact that $E^{d-k,t+k-1}_k=E^{d+k,t-k+1}_k=0$  for all $k\geq 2$, we have $E^{d,t}_{\infty}=E^{d,t}_2$.  Hence $\operatorname{H}^d_J (H^{t}_I(M)) \cong \operatorname{H}^{d+t}_{I+J}(M)$.\\
Since $\operatorname{H}^d_J (\operatorname{H}^{t}_I(M)) \neq 0,$ it follows from Grothendieck's vanishing theorem that $\dim\operatorname{Supp}(\operatorname{H}^{t}_I(M))\geq d$. On the other hand, since $\dim\operatorname{Supp}(\operatorname{H}^{t}_I(M))\leq \dim(R/{I+\operatorname{Ann}M})=d$, we conclude that $\dim\operatorname{Supp}(\operatorname{H}^{t}_I(M))= d$.
\end{proof}
So far, for a finitely generated $R$- module $M$, the best known lower bound for $\operatorname{cd}(I,M)$ is $\operatorname{ht}_M(I)=\operatorname{ht}I(R/{\operatorname{AnnM}})$.  As an immediate consequence of Theorem \ref{theorem3.1}, we sharpen this bound to $\dim(M)-\dim(M/{IM})\geq \operatorname{ht}_M(I)$.
\begin{corollary}\label{corollary3.2} Let $(R,\mathfrak m)$ be a Noetherian local ring,  $M$ a finitely generated $R$-module of dimension $n$ and $I$ an ideal of $R$ such that $\dim(M/IM)=d$. Then $ n-d$ is a lower bound for $c=\operatorname{cd}(I,M)$. Moreover, if $c=n-d$, then 
\begin{equation*}
\operatorname{H}^{d}_{\mathfrak m} (\operatorname{H}^{n-d}_I(M)) \cong \operatorname{H}^{n}_{\mathfrak m}(M)
\end{equation*}
and $\dim \operatorname{Supp}(\operatorname{H}^{n-d}_I(M))= d$.
\end{corollary}
\begin{proof}This follows from Theorem \ref{theorem3.1} and the fact that $H^{n}_{\mathfrak m}(M)\neq 0$. 
\end{proof}


For an ideal $I$ of $R,$ it is a well-known fact that $\operatorname{ht}(I)\leq \operatorname{cd}(I)\leq \operatorname{ara}(I)$, where $\operatorname{ara}(I)$ denotes the smallest number of elements of R required to generate $I$ up 
to radical. If, in particular, $\operatorname{ara}(I)=\operatorname{cd}(I)=\operatorname{ht}(I)$, then $I$ is called a \textit{set-theoretic complete intersection ideal}. Determining set-theoretic complete intersection ideals is a classical and long-standing problems in commutative algebra and 
algebraic geometry. Many questions related to an ideal $I$ to being a set-theoretic complete intersection are still open, see \cite{L} for more details.\
Varbaro in \cite{Varbaro} show that under certain conditions there exists ideals $I$ satisfying the property that $\operatorname{cd}(I)=\operatorname{ht}(I)$, knowing the existence of ideals with such properties, we have the following: \\ 
\begin{corollary}\label{corollary3.3} Let $(R,\mathfrak m)$ be a Noetherian local ring of dimension $n$ and $I$ an ideal of $R$ with $d=dim(R/I)$ such that $\operatorname{cd}(I)=\operatorname{ht}(I)=h$. Then $\dim(R)=\operatorname{ht}(I)+\dim(R/I)$ and
\begin{equation*}
\operatorname{H}^{n-h}_{\mathfrak m} (\operatorname{H}^{h}_I(R)) \cong \operatorname{H}^{n}_{\mathfrak m}(R).
\end{equation*}
\end{corollary}
\begin{proof} It follows from Corollary \ref{corollary3.2} that $ \dim(R)-\dim(R/I)\leq \operatorname{cd}(I)=\operatorname{ht}(I)$, while the other side of the inequality always holds. Therefore $\dim(R)=\operatorname{ht}(I)+\dim(R/I)$. Now the required isomorphim follows from Corollary  \ref{corollary3.2}.
\end{proof}
We end this section with the following conclusion :
\begin{corollary} Let $(R,\mathfrak m)$ be a non-catenary Noetherian local domain of dimension $n$. Then there is at least one prime ideal of $R$ that is not a set theoretic complete intersection.
\end{corollary}
\begin{proof} Since $R$ is non-catenary, there is a prime ideal $\mathfrak{p}$ of $R$ such that $\operatorname{ht}(\mathfrak{p})<n-\dim(R/{\mathfrak{p}})$. Then it follows from Corollary \ref{corollary3.3} that $\operatorname{cd}(\mathfrak{p})\neq \operatorname{ht}(\mathfrak{p})$ and therefore ${\mathfrak{p}}$ can not be a set theoretic complete intersection ideal.
\end{proof}
\section{Artinianness and Non-Artinianness of Top Local Cohomology Modules}
Let $(R,\mathfrak m)$ be a Noetherian local ring and $M$ an $R$-module of dimension $n$. Recall that if $M$ is a coatomic or a weakly finite (in particular, finitely generated, $I$-cofinite, or a balanced big Cohen Macaulay) module, then $\operatorname{H}^n_{\mathfrak m}(M)$ is nonzero and Artinian [\cite{Aghapournahr-Melkersson}, \cite{Bagheri}].\\
In light of this information, we have the following results the first of which is the generalization of  Theorem 7.1.6 of \cite{BS}:
\begin{theorem}\label{theorem4.1}  Let $(R,\mathfrak m)$ be a Noetherian local ring and $M$ an $R$-module of dimension $n$ such that $\operatorname{H}^{n}_{\mathfrak m}(M)$ is Artinian. Then $\operatorname{H}^n_I(M)$ is Artinian for all ideals $I$ of $R$.
\end{theorem}
\begin{proof} We use induction on $d=\dim(R/I)$. If $d=0$, then $I$ is an $\mathfrak m$-primary ideal and so $\operatorname{H}^n_I(M)\cong \operatorname{H}^n_{\mathfrak m}(M)$ is Artinian.\\
Let now $d>0$ and suppose the hypothesis is true for all ideals $J$ of $R$ with $\dim(R/J)<d$.\\
Choose an element $x\in \mathfrak m$ such that  $\dim(R/{(I+Rx)})=d-1<d.$ Then by induction hypothesis, $\operatorname{H}^n_{I+Rx}(M)$ is Artinian. Now consider the exact sequence 
\begin{equation*}\xymatrix{\cdots \ar[r]& \operatorname{H}^n_{I+Rx}(M)\ar[r]&{\operatorname{H}^{n}_{I} (M)}\ar[r]& \operatorname{H}^{n}_{I} (M_x) \ar[r]& \cdots}
\end{equation*}
Since  $\operatorname{H}^n_{I+Rx}(M)$ is Artinian and $\operatorname{H}^{n}_{I} (M_x)=0$ ($\dim(M_x)<n$, as $x\in \mathfrak m$), it follows from the above exact sequence that $\operatorname{H}^{n}_{I} (M)$ is Artinian.
\end{proof}

Recall that a class $\mathcal S$ of $R$-modules is a Serre subcategory of the category of $R$-modules, $\mathcal{C}(R)$, when it is closed under taking submodules, quotients and extensions.
The main result of this section is the following:
\begin{theorem}\label{theorem4.2}Let $R$ be a Noetherian ring, $M$ an $R$-module (not necessarily finitely generated)  and let $\mathcal{S}$ be a Serre subcategory of $\mathcal{C}(R)$. Let $I\ and\  J$ be two ideals of $R$ such that $\operatorname{H}^{t+i}_J(\operatorname{H}^{c-i}_I(M))\in \mathcal{S}$ for all $0< i\leq c=\operatorname{cd}(I,M)$ and $\operatorname{H}^{t+c}_{I+J}(M)\notin \mathcal{S}$ for some positive integer $t$. Then $\operatorname{H}^t_J(\operatorname{H}^c_I(M))\notin \mathcal{S}. $
\end{theorem}
\begin{proof} Consider the Grothendieck's spectral sequence 
\begin{equation*}
 E^{p,q}_2=\operatorname{H}^p_{J}(\operatorname{H}^q_I(M))  \Longrightarrow \operatorname{H}^{p+q}_{I+J}(M)
\end{equation*}
and look at the stage $p+q=c+t$. Let now $0< i\leq c=\operatorname{cd}(I,M).$ Since $E^{{t+i},c-i}_{\infty}= E^{{t+i},c-i}_{r}$ for sufficiently large $r$ and $E^{{t+i},c-i}_{r}$  is a subquotient of $E^{{t+i},c-i}_2\in \mathcal{S}$, $E^{{t+i},c-i}_{\infty}\in \mathcal{S}$ for all $0< i\leq c=\operatorname{cd}(I,M)$.\\
On the other hand, since $E^{t,c}_2=\operatorname{H}^{t}_{J}(\operatorname{H}^c_I(M))  \Longrightarrow \operatorname{H}^{t+c}_{I+J}(M)$, there exists a finite filtration
\begin{equation*}{0}=\Phi^{t+c+1}\operatorname{H}^{t+c}\subseteq \Phi^{t+c}\operatorname{H}^{t+c}\subseteq \cdots\subseteq  \Phi^{1}\operatorname{H}^{t+c}\subseteq \Phi^{0}\operatorname{H}^{t+c}=\operatorname{H}^{t+c}
\end{equation*}
of $\operatorname{H}^{t+c}=\operatorname{H}^{t+c}_{I+J}(M)$ such that $E^{p,q}_{\infty}={{\Phi^{p}\operatorname{H}^{t+c}}/{\Phi^{p+1}\operatorname{H}^{t+c}}}$ for all $p+q=t+c$.  Since for all $p< t$, $E^{p,q}_{\infty}=0$, we have that $\Phi^{t}\operatorname{H}^{t+c}=\cdots =\Phi^{1}\operatorname{H}^{t+c}=\Phi^{0}\operatorname{H}^{t+c}=\operatorname{H}^{t+c}$. But then since $E^{{t+i},c-i}_{\infty}={{\Phi^{t+i}\operatorname{H}^{t+c}}/{\Phi^{t+i+1}\operatorname{H}^{t+c}}}\in \mathcal{S}$ for all $0<i\leq c$, $\Phi^{t+1}\operatorname{H}^{t+c}\in \mathcal{S}$ and so it follows from the short exact sequence
\begin{equation*}\xymatrix{0\ar[r]& {\underbrace{\Phi^{t+1}\operatorname{H}^{t+c}}_{\in \mathcal{S}}}\ar[r]&{\underbrace{\operatorname{H}^{t+c}_{I+J}(M)}_{\notin \mathcal{S}}}\ar[r]& E^{t,c}_{\infty}\ar[r]&0}
\end{equation*} 
that $E^{t,c}_{\infty}\notin \mathcal{S}$. Since $E^{t,c}_{\infty}$ is a subquotient of $E^{t,c}_2$ and $E^{t,c}_{\infty}\notin \mathcal{S}$, it follows that $E^{t,c}_2=\operatorname{H}^{t}_J({\operatorname{H}^c_{I}(M)})\notin \mathcal{S}.$
\end{proof}
\begin{corollary}\label{corollary4.3}Let $(R,\mathfrak m)$ be a Noetherian local ring,  $M$ an $R$-module of dimension $n$ such that $\operatorname{H}^n_{\mathfrak m}(M)\neq 0$ with $\dim(R/{I+AnnM})=d$. If $\operatorname{H}^{\operatorname{cd}(I,M)}_I(M)$ is Artinian, then either $\operatorname{cd}(I,M)=n$ or $ \operatorname{H}^{n-i}_{\mathfrak m}(\operatorname{H}^{i}_I(M))\neq 0$ for some $n-d\leq i< \operatorname{cd}(I,M).$
\end{corollary}
\begin{proof} We prove the contrapositive of the statement. Let $\mathcal{S}$ be the category of zero module and suppose that $c=\operatorname{cd}(I,M)<n$ and $\operatorname{H}^{n-i}_{\mathfrak m}(\operatorname{H}^{i}_I(M))=0\in \mathcal{S}$ for all $n-d\leq i< c.$ But then since $\operatorname{H}^n_{\mathfrak m}(M)\notin \mathcal{S}$, it follows from Theorem \ref{theorem4.2} that $\operatorname{H}^{n-c}_{\mathfrak m}(\operatorname{H}^{c}_I(M))\neq 0$. Hence $\dim \operatorname{Supp}(\operatorname{H}^{c}_I(M))>0$ and so $\operatorname{H}^{c}_I(M)$ is not Artinian.
\end{proof}
The following results determine the Artinianness and non-Artinianness of the top local cohomology module, $\operatorname{H}^{\operatorname{cd}(I,M)}_I(M)$, for the ideals of small dimension.
\begin{theorem}\label{theorem4.4} Let $(R,\mathfrak m)$ be a Noetherian local ring,  $M$ an $R$-module of dimension $n$ such that $\operatorname{H}^{n}_{\mathfrak m}(M)$ is nonzero and Artinian and let $I$ be an ideal of $R$ such that $\dim(R/{I+\operatorname{Ann}M})=1$. Then $\operatorname{H}^{\operatorname{cd}(I,M)}_I(M)$ is Artinian if and only if $\operatorname{cd}(I,M)=n.$ 
\end{theorem}
\begin{proof} Since $\dim(R/{I+\operatorname{Ann}M})=1$ and $\operatorname{H}^n_{\mathfrak m}(M)\neq 0$ , it follows from Theorem \ref{theorem3.1} that either $\operatorname{cd}(I,M)=n-1$ or $\operatorname{cd}(I,M)=n.$ If $\operatorname{cd}(I,M)=n$, then by Theorem \ref{theorem4.1}, $\operatorname{H}^{n}_I(M)$ is Artinian. If, on the other hand, $\operatorname{cd}(I,M)=n-1$, then it follows from Theorem \ref{theorem3.1} that $\dim\operatorname{Supp}(\operatorname{H}^{n-1}_I(M))=1$ and so $\operatorname{H}^{n-1}_I(M)$ is non-Artinian.
\end{proof}

\begin{theorem}\label{theorem4.5} Let $(R,\mathfrak m)$ be a Noetherian local ring, $M$ an $R$-module of dimension $n$ such that $ \operatorname{H}^n_{\mathfrak m}(M)$ is nonzero and Artinian and let $I$ be an ideal of $R$ with $\dim(R/{I+\operatorname{Ann}(M)})=2$. If $\operatorname{H}^{\operatorname{cd}(I,M)}_I(M)$ is Artinian, then either $\operatorname{cd}(I,M)=n$, or $\operatorname{cd}(I,M)=n-1$ and $\operatorname{H}^{2}_{\mathfrak m}(\operatorname{H}^{n-2}_I(M)) \neq 0$.
\end{theorem}
\begin{proof} Since $\dim(R/{I+\operatorname{Ann}M})=2$ and $\operatorname{H}^n_{\mathfrak m}(M)\neq 0$, it follows from Theorem \ref{theorem3.1} that $n-2$ is a lower bound for $\operatorname{cd}(I,M)$.  If $\operatorname{cd}(I,M)=n-2$, then again by Theorem \ref{theorem3.1}, $\dim\operatorname{Supp}(\operatorname{H}^{n-2}_I(M))=2$ and so $\operatorname{H}^{\operatorname{cd}(I,M)}_I(M)$ is non-Artinian. If, on the other hand,  $\operatorname{cd}(I,M)=n$, then  from Theorem \ref{theorem4.1}, $\operatorname{H}^{\operatorname{cd}(I,M)}_I(M)$ is Artinian. Finally, if  $\operatorname{cd}(I,M)=n-1$ and $\operatorname{H}^{n-1}_I(M)$ is Artinian, then the result follows from Corollary \ref{corollary4.3}. 
\end{proof}
\footnotesize{{}
\end{document}